\documentclass[a4paper,11pt]{amsart}

\usepackage{mathpazo}
\usepackage{amssymb}
\usepackage[ansinew]{inputenc}
\usepackage[pagebackref,
    ,pdfborder={0 0 0}
    ,urlcolor=black,a4paper,hypertexnames=false]{hyperref}
\hypersetup{pdfauthor={Daniel Fauser},
	pdftitle= Integral foliated simplicial volume and S1-actions}

\usepackage{pgf}
\usepackage{tikz}
\usepackage{tikz-cd}
\usetikzlibrary{decorations.pathmorphing}
\usepackage[all]{xy}
\SelectTips{cm}{}

\newtheorem{thm}{Theorem}[section]
\newtheorem{prop}[thm]{Proposition}
\newtheorem{lem}[thm]{Lemma}
\newtheorem{cor}[thm]{Corollary}

\theoremstyle{remark}
\newtheorem{rem}[thm]{Remark}
\newtheorem{exa}[thm]{Example}

\newtheorem{setup}[thm]{Setup}

\theoremstyle{definition}
\newtheorem{defi}[thm]{Definition}

\usepackage{amsfonts}
\newcommand{\Z}{\mathbb{Z}}

\newcommand{\R}{\mathbb{R}}
\newcommand{\N}{\mathbb{N}}

\DeclareMathOperator{\id}{id}

\DeclareMathOperator{\rk}{rk}
\DeclareMathOperator{\rg}{rg}
\DeclareMathOperator{\tors}{tors}
\DeclareMathOperator{\map}{map}
\DeclareMathOperator{\res}{res}

\DeclareMathOperator{\const}{const}
\def\epsilon{\varepsilon}

\DeclareMathOperator{\sign}{sign}

\def\sv#1{\|#1\|}

\makeatletter
\newcommand\norm{\bBigg@{0.8}}

\makeatother
\newcommand{\ifsv}[2][norm]{\csname #1l\endcsname\bracevert\!#2\!%
                            \csname #1r\endcsname\bracevert}
\newcommand{\stisv}[2][norm]{\indnorml[#1]{#2}{\Z}{\infty}}
\newcommand{\indnorml}[4][flex]{\csname #1l\endcsname\|#2%
                                 \csname #1r\endcsname\|_{#3}^{#4}\mathclose{}}

\def\ucov#1{%
  \widetilde{#1}
}

\usepackage{color}
\usepackage{pdfcolmk}

\author{Daniel Fauser}
\title[Integral foliated simplicial volume and $S^1$-actions]%
      {Integral foliated simplicial volume and $S^1$-actions}
\date{\today.\ \copyright{\ D.~Fauser 2017}.
	This work was supported by the CRC~1085 \emph{Higher Invariants} (Universit\"at Regensburg, funded by the DFG).
}

\begin{document}

\phantom{.}
\vspace{-.1\baselineskip}

\maketitle

\begin{abstract}
	The simplicial volume of
	oriented closed connected smooth manifolds that admit a non-trivial smooth $S^1$-action vanishes.
	In the present work we prove a version of this result for
	the integral foliated simplicial volume of aspherical manifolds:
	The integral foliated simplicial volume of aspherical oriented closed
	connected smooth manifolds that admit a non-trivial smooth $S^1$-action vanishes.
	Our proof uses the geometric construction of Yano's proof for ordinary simplicial volume
	as well as the parametrised uniform boundary condition for~$S^1$.
\end{abstract}

\section{Introduction}

It is a long standing question of Gromov whether all $L^2$-Betti numbers of an aspherical
oriented closed connected manifold with trivial simplicial volume are zero~\cite[p.~232]{gromov}.
For such manifolds with non-trivial $S^1$-action it is known that all $L^2$-Betti numbers 
vanish~\cite[Corollary~1.43]{lueck}.
Moreover, Gromov and Yano independently showed that the simplicial volume of oriented closed connected smooth manifolds with non-trivial smooth $S^1$-action is zero~\cite{vbc,yano}.
The integral foliated simplicial volume (see Subsection~\ref{subsec:ifsv}) yields an upper bound for the
$L^2$-Betti numbers as well as for the ordinary simplicial volume~\cite{schmidt}.
This leads to the question whether the integral foliated simplicial volume of an aspherical 
oriented closed connected smooth manifold with non-trivial smooth $S^1$-action is also zero.
In this work we prove the following result, which answers this question in the positive.

\begin{thm}[integral foliated simplicial volume and $S^1$-actions]
\label{thm:mainthm}
	Let~$M$ be an oriented compact connected smooth 
	manifold that admits a smooth $S^1$-action without fixed points
	such that the inclusion of every orbit into~$M$ is $\pi_1$-injective.
	Then~$\ifsv{M,\partial M} =0$.
	
	More precisely: If~$\alpha\colon \pi_1 (M) \curvearrowright (Z,\mu)$ is an essentially free
	standard $\pi_1(M)$-space, then~$\ifsv{M,\partial M}^\alpha =0$.
\end{thm}

\begin{cor}
\label{cor:maincor}
	Let~$M$ be an aspherical oriented closed connected smooth manifold
	that admits a non-trivial smooth $S^1$-action.
	Then~$\ifsv M =0$.
	
	More precisely: If~$\alpha\colon \pi_1 (M) \curvearrowright (Z,\mu)$ is an essentially free
	standard $\pi_1(M)$-space, then~$\ifsv{M}^\alpha =0$.
\end{cor}

Corollary~\ref{cor:maincor} directly follows from Theorem~\ref{thm:mainthm}
and a result on the structure of non-trivial $S^1$-actions on aspherical closed manifolds of 
L\"uck~\cite[Corollary~1.43]{lueck}.

For the proof of Theorem~\ref{thm:mainthm}, we will combine Yano's construction with
the parame\-trised uniform boundary condition for~$S^1$~\cite{varubc}.
Note that one can also drop the condition on the $S^1$-action to be fixed point
free in Theorem~\ref{thm:mainthm}~\cite[Theorem~4.2.1]{thesis} but we keep it for 
convenience in this work.

Sauer established an upper bound on the $L^2$-Betti numbers of aspherical manifolds
in terms of minimal volume~\cite[p.~3]{sauer} 
and Braun improved this result to boundedness from above of the 
integral foliated simplicial volume in terms of minimal volume~\cite[Corollary~9]{braun}.
Since the minimal volume of compact smooth manifolds
with locally free $S^1$-action is zero~\cite[Appendix~2]{vbc}
the result of Braun yields an alternative proof of
Corollary~\ref{cor:maincor}~\cite[Corollary~5.6]{braun}.

\subsection*{Applications}

Apart from having a new proof for vanishing of all $L^2$-Betti numbers of aspherical manifolds with non-trivial smooth $S^1$-action, Theorem~\ref{thm:mainthm} yields applications to gradient invariants: We have a new approximation result for simplicial volume and hence vanishing results for the Betti number gradient, the torsion homology gradient, the rank gradient and the cost
of (fundamental groups of) aspherical manifolds with non-trivial smooth $S^1$-action.

We first recall some definitions:
The \emph{stable integral simplicial volume} of an oriented closed connected manifold~$M$ with fundamental group~$\Gamma$ is given by
\[ \stisv M := \inf_{\Lambda\in F(\Gamma)} \frac{\sv{\widetilde{M} / \Lambda}_\Z}{[\Gamma:\Lambda]},
\]
where~$F(\Gamma)$ is the set of all finite index subgroups of~$\pi_1(M)$.
A \emph{residual chain} in a finitely generated group~$\Gamma$ is a descending
sequence~$\Gamma = \Gamma_0 > \Gamma_1> \dots$ of normal finite index subgroups
whose intersection is trivial.

\begin{cor}[stable integral simplicial volume and $S^1$-actions]
	Let~$M$ be an aspherical oriented closed connected smooth manifold
	with residually finite fundamental group~$\Gamma$ that admits a
	non-trivial smooth $S^1$-action. Then, we have
	\[ \stisv M = 0.
	\]
	
	More generally: If~$(\Gamma_n)_{n\in \N}$ is a \emph{Farber chain}~\cite{abertnikolov}
	(e.g., a residual chain) in~$\Gamma$, then
	\[ \inf_{n\in N} \frac{\sv{M_n}_\Z}{[\Gamma:\Gamma_n]} = 0,
	\]
	where~$M_n \longrightarrow M$ denotes the covering of~$M$ associated
	to~$\Gamma_n$ for all~$n\in \N$ and~$\sv{\cdot}_\Z$ is the integral simplicial volume.
\end{cor}

\begin{proof}
	A Farber chain in a finitely generated group~$\Gamma$ yields a standard
	$\Gamma$-space that is essentially free~\cite[Section~3]{abertnikolov}.
	Now, the result follows from the relation between integral foliated simplicial volume and 
	stable integral simplicial volume~\cite[Theorem~2.6]{flps} and Corollary~\ref{cor:maincor}.
\end{proof}

Let $\tors A$ denote the torsion of a finitely generated abelian group~$A$ and~$\rk_R$
denote the $R$-dimension of the free part of finitely generated $R$-modules.
The \emph{rank gradient} of a residually finite group~$\Gamma$ (with respect to a Farber chain~$(\Gamma_n)_{n\in \N}$ in~$\Gamma$) is defined as
\[ \rg(\Gamma) := \inf_{\Lambda \in F(\Gamma)} \frac{d(\Lambda)-1}{[\Gamma:\Lambda]}
\quad \text{and} \quad
		\rg\bigl(\Gamma,(\Gamma_n)_{n\in \N}\bigr) := 
		\inf_{n\in N} \frac{d(\Gamma_n)-1}{[\Gamma:\Gamma_n]}
\]
respectively, where~$d(\cdot)$ denotes the minimal number of generators of a finitely generated group.

\begin{cor}[gradient invariants and $S^1$-actions]
\label{cor:app2}
	Let~$M$ be an aspherical oriented closed connected smooth manifold
	with residually finite fundamental group~$\Gamma$
	that admits a non-trivial smooth $S^1$-action.
	Let~$(\Gamma_n)_{n\in \N}$ be a Farber chain in~$\Gamma$ and let~$M_n \longrightarrow M$
	denote the covering of~$M$ associated to~$\Gamma_n$ for all~$n\in \N$.
	Then for all~$k\in \N$ and for every principal ideal domain~$R$, we have
	\begin{align*}
		\limsup_{n\rightarrow \infty} \frac{\rk_R H_k(M_n;R)}{[\Gamma:\Gamma_n]}=0
		\quad\text{and}\quad
		\limsup_{n\rightarrow \infty} \frac{\log |\tors H_k(M_n;\Z)|}{[\Gamma:\Gamma_n]}=0.
	\end{align*}
	Moreover, we have
	\[ \rg\bigl(\Gamma,(\Gamma_n)_{n\in \N}\bigr)=0;
	\]
	in particular, $\rg(\Gamma)=0$ and~$\mathrm{cost}(\Gamma) = 1$.
\end{cor}

\begin{proof}
	The first part of the corollary follows from a result on homology bounds by
	Frigerio, L\"oh, Pagliantini and Sauer~\cite[Theorem~1.6]{flps}.
	The second part follows from the fact that stable integral simplicial volume yields an upper
	bound for the rank gradient~\cite[Theorem~1.1]{loeh} 
	as well as for the cost~\cite[Theorem~1.2]{loehcost}.
\end{proof}

With similar methods as in the present work one can prove a variant of Theorem~\ref{thm:mainthm} 
for $S^1$-bundles~\cite[Proposition~4.2]{graphmfds}\cite[Theorem~4.3.1]{thesis}. 
This variant can be used to prove vanishing of the stable integral simplicial volume of 
generalised graph manifolds in the sense of Friedl, L\"oh and the 
author (e.g. non-spherical 3-dimensional graph manifolds)~\cite[Theorem~1.2]{graphmfds}.

\subsection*{Organisation of this article}

We briefly recall Yano's construction in Section~\ref{sec:setup}.
In Section~\ref{sec:ubc} we introduce relative integral foliated simplicial volume
and the parametrised uniform boundary condition for~$S^1$.
Additional prerequisites are provided in Section~\ref{sec:prep}.
In Section~\ref{sec:constructchain} we will construct
parametrised chains of small norm that we adjust
in Section~\ref{sec:proof} to get para\-metrised 
fundamental cycles of small norm.

\subsection*{Acknowledgements}

I would like to thank my advisor Clara L\"oh for her guidance
and all the helpful discussions and suggestions.


\section{Yano's construction}
\label{sec:setup}
In this section, we briefly recall the definitions and 
notations of Yano's construction~\cite[Section~2]{yano}\cite[Section~4.2]{thesis}
that he used for the proof of vanishing of the simplicial volume of 
smooth manifolds with non-trivial smooth $S^1$-action.

Let~$M$ be an oriented closed connected smooth $n$-manifold that
admits a smooth $S^1$-action without fixed points. 
Using so-called hollowings, Yano defines a sequence
\begin{align*}
	M_{n-2} \stackrel{p_{n-3}}{\longrightarrow} M_{n-3} \stackrel{p_{n-4}}{\longrightarrow} \dots \stackrel{p_1}{\longrightarrow} M_1 \stackrel{p_{0}}{\longrightarrow} M_0 = M
\end{align*}
of compact manifolds with smooth $S^1$-action and $S^1$-equivariant maps such that~$M_{n-2}$ 
splits as~$N\times S^1$ with~$N$ an oriented compact connected manifold (possibly) with boundary.
Note that the assumption on the $S^1$-action to be fixed point free allows us
to skip the first~$n-1$ steps in the sequence that Yano defined originally.

The idea of Yano's proof then is the following:
We know that the (relative) simplicial volume of~$M_{n-2}\cong N \times S^1$ is zero.
So we can choose a fundamental cycle of~$M_{n-2}$ of small $\ell^1$-norm.
The pushforward of this relative cycle to~$M$ unfortunately is in general 
no cycle in~$M$ any more, but one can adjust this pushforward by fillings to get 
a fundamental cycle of~$M$ without changing the norm too much.

One can easily generalise Yano's construction to compact manifolds with boundary by
allowing hollowings not just transversal to but also along the boundary
(which can be realised via doubling techniques~\cite[Section~B.3]{thesis}).

We come back to the sequence of manifolds above (where we allow~$M$ to be compact,
possibly with boundary):
First, we choose a triangulation on~$\overline M$ as follows:
For all~$r\in \N_{\geq 2}$ let~$L_r\subset M$ be the set of points whose stabilisers
contain the set
\[ \biggl\{0, \frac{1}{r},\dots, \frac{r-1}{r}\biggr\}
			\subset S^1 \cong \R/\Z
\]
and let~$L:= \bigcup_{r\in \N_{\geq 2}} L_r$. Without loss of generality we may assume that
the $S^1$-action is \emph{effective}, i.e., there exists no element in~$S^1$ that fixes every
point in~$M$. Then every~$L_r$ is a smooth submanifold of~$M$
of dimension less than~$n-1$ that admits a smooth $S^1$-action.
Let~$\pi\colon M\longrightarrow \overline M$ be the projection of~$M$ onto its orbit space.
Then,~$\overline M$ admits an equivariant triangulation~\cite[Corollary~3.8]{verona} 
such that for all~$r\in \N_{\geq 2}$ the orbit space~$\overline L_r$ of~$L_r$ is a subcomplex
of the triangulation.

Now, we want to extend the sequence by~$M_{-1}:=M$ and~$p_{-1}:= \id_M$
for notational reasons concerning the case with boundary.
We inductively define the map $p_j$ for all~$j\in \{0,\dots, n-3\}$ to be an 
\emph{equivariant hollowing
at~$X_j \subset M_j$}, i.e., $M_{j+1}$ is obtained from~$M_j$ by removing an equivariant open 
tubular neighbourhood~$T$ of~$X_j$ and glueing in a collar; then~$p_j\colon M_{j+1} \longrightarrow M_j$ is a projection that maps the collar to~$T$ in a way that the boundary of 
the collar is mapped to~$X_j$.
Here,~$X_j$ is the pullback of the $j$-skeleton of
the orbit space~$\overline M$ along~$\pi \circ p_0\circ \cdots \circ p_{j-1}$.
We write
\[ p_{l,l'}:=p_{l'}\circ \dots \circ p_{l-1}\colon M_l \longrightarrow M_{l'},
\]
for all~$l,l' \in \{ -1,0,\dots, n-2\}$ with~$l'< l$.
We set~$X_{-1}:=\partial M \subset M_{-1}$.
For all~$j\in \{-1, \dots n-3\}$ the \emph{hollow wall of~$p_j$} is given by
\[ N_j := p_j^{-1}(X_j) \subset M_{j+1}
\]
and~$\widetilde{N}_j\subset M_{n-2}$ is the pullback
of~$N_j$ along~$p_{n-2,j}$.
Let~$k\in \{1, \dots, n-1\}$ and let~$j_1, \dots, j_k \in \{-1, \dots, n-3\}$ be pairwise distinct.
We define
\[ \widetilde{N}_{j_1, \dots, j_k} := \widetilde{N}_{j_1} \cap \dots \cap \widetilde{N}_{j_k} \subset M_{n-2}
\]
and
\[ X_{j_1, \dots, j_k} := p_{n-2, j_1} (\widetilde{N}_{j_1, \dots, j_k}) \subset M_{j_1}.
\]
We write
\[ \overline N_{j_1, \dots, j_k} \quad \text{and} \quad \overline X_{j_1, \dots, j_k}
\]
for the orbit space of~$\widetilde{N}_{j_1, \dots, j_k}$ and~$X_{j_1, \dots, j_k}$ respectively.
Then, Yano shows the 
following~\cite[Lemma~4, Lemma~6 and Lemma~7]{yano}\cite[Lemma~4.2.3 and Lemma~4.2.4]{thesis}:

\begin{lem}
\label{lem:trivialbundle}
	Let~$k\in \{1, \dots, n-2\}$ and let~$j_1, \dots, j_k \in \{0, \dots, n-3\}$ be pairwise
	distinct. Then, each connected component of the orbit space~$\overline X_{j_1,\dots, j_k}$
	of~$X_{j_1, \dots, j_k}$ is contractible and we have
	\[ X_{j_1,\dots, j_k} \cong \overline X_{j_1,\dots, j_k} \times S^1
			\quad \text{and} \quad
		 M_{n-2} \cong \overline M_{n-2} \times S^1.
	\]
\end{lem}

For the case with boundary we need in addition the following observation:

\begin{prop}\label{prop:boundary}
For all pairwise distinct~$j_1, \dots, j_k \in \{0, \dots, n-3\}$
we have that~$X_{j_1, \dots, j_k, -1}$ is the union of
the connected components~$Y$ in $X_{j_1,\dots, j_k}$ with
\[ Y \subset p_{n-2, j_1} (\widetilde N_{-1}).
\]
\end{prop}
\begin{proof}
	Let~$j\in \{0,\dots, n-3\}$.
	We only show the statement for~$X_{j,-1}\subset X_j$.
	The general case can be proven similarly.
	Let~$Y\subset \overline X_j$ be a connected component.
	As in Yano's proof of Lemma~\ref{lem:trivialbundle} we observe that~$Y$ is homeomorphic
	to~$\Delta^j_j$, where~$\Delta_j^j$ is obtained from the standard simplex~$\Delta^j$ by
	hollowing inductively along the l-skeleton for all~$l\in \{0,\dots, j-1\}$.
	From this it follows easily that we are in one of the following cases:
	\begin{enumerate}
		\item We have~$Y\subset \overline p_{n-2,j}(\overline N_{-1})$, or
		\item we have~$Y \cap \overline p_{n-2,j}(\overline N_{-1})= \emptyset$.
	\end{enumerate}
	In the first case, we have
	\begin{align*}
		Y &\subset \overline X_j \cap \overline p_{n-2,j}(\overline N_{-1})
		  = \overline p_{n-2,j}(\overline N_{j}) \cap \overline p_{n-2,j}(\overline N_{-1})\\
			&\subset \overline p_{n-2,j}(\overline N_{j,-1})
			= \overline X_{j,1},
	\end{align*}
	where the last inclusion follows from
	\[ \overline p_{n-2,j}(\overline N_{j}\setminus\overline N_{-1})\cap
				\overline p_{n-2,j}(\overline N_{-1}\setminus \overline N_{j}) = \emptyset
	\]
	which holds by construction~\cite[Lemma~4.2.8]{thesis}.
	In the second case, we have~$Y\cap \overline X_{j,-1}=\emptyset$.
\end{proof}

\begin{rem}
\label{rem:complex}
It follows that each~$\widetilde{N}_{j_1, \dots, j_k}$ decomposes
as~$\overline N_{j_1, \dots, j_k}\times S^1$. We choose a simplicial structure
on~$\partial \overline M_{n-2}$ that is compatible with the decompositions
\[ \partial \overline M_{n-2}= \bigcup_{i=-1}^{n-3} \overline N_i
		\quad \text{ and } \quad
	 \partial \overline N_{j_1,\dots, j_k} = \bigcup_{j} \overline N_{j_1,\dots, j_k,j},
\]
where~$j$ ranges over~$\{-1,\dots, n-3\}\setminus\{j_1,\dots, j_k\}$.
Then each~$\overline N_{j_1, \dots, j_k}$ is an $(n-2-k)$-dimensional subcomplex
of the $(n-2)$-dimensional complex~$\partial \overline M_{n-2}$.
\end{rem}


\section{Relative integral foliated simplicial volume and the parametrised uniform boundary condition}
\label{sec:ubc}


\subsection{The $\ell^1$-norm on the singular chain complex and simplicial volume}
We recall the definition of the relative simplicial volume introduced by Gromov~\cite{vbc}.

\begin{defi}[$\ell^1$-norm on the singular chain complex]
	Let~$R\in \{\Z,\R\}$. Let~$M$ be a topological space and let~$n\in \N$.
	For a singular chain~$c= \sum_{j=1}^k a_j \cdot \sigma_j \in C_n(M;R)$
	written in reduced form (i.e., the singular simplices~$\sigma_1, \dots, \sigma_k$
	are pairwise distinct) we define the \emph{$\ell^1$-norm of~$c$} by
	\[ |c|_1 := \sum_{j=1}^k |a_j|.
	\]
\end{defi}

\begin{rem}[functoriality]
\label{rem:functorial}
	The $\ell^1$-norm on the singular chain complex is \emph{functorial} in the following sense:
	Let~$f\colon M\longrightarrow N$ be a continuous map between topological spaces~$M$ and~$N$.
	Then~$\|C_n(f;R)\| \leq 1$, where~$\|\cdot\|$ denotes the operator norm.
\end{rem}

\begin{defi}[relative simplicial volume]
	Let~$M$ be an oriented compact connected $n$-manifold.
	A \emph{relative $\R$-fundamental cycle of~$M$} is a chain~$c\in C_n(M;\R)$ of the form
	\[ c = c_\Z + \partial b + d
	\]
	where $c_\Z \in C_n(M;\Z) \subset C_n(M;\R)$ is an ordinary relative fundamental cycle of~$M$,
	$b\in C_{n+1}(M;\R)$ and~$d\in C_{n}(\partial M;\R) \subset C_n(M;\R)$.
	In other words, $c$ is a cycle in~$C_n(M,\partial M;\R)$ representing~$[M, \partial M]_\R$.
	Then the \emph{relative simplicial volume of~$M$} is defined by
	\[ \|M,\partial M\| := 
				\inf \bigl\{ |c|_1 \bigm| c\in C_n(M,\R) \text{ represents } [M,\partial M]_\R \bigr\}.
	\]
	If~$\partial M = \emptyset$, we write~$\|M\|:=\|M,\partial M\|$.
\end{defi}


\subsection{The parametrised $\ell^1$-norm}

The parametrised $\ell^1$-norm is given as the $\ell^1$-norm on the singular chain complex with
twisted coefficients that are induced by actions of the fundamental group on probability spaces.
This leads to the (relative) integral foliated simplicial volume (Subsection~\ref{subsec:ifsv}).

\begin{defi}[standard $\Gamma$-space]
	Let~$\Gamma$ be a countable group.
	A \emph{standard $\Gamma$-space}~$\alpha = \Gamma \curvearrowright (Z,\mu)$ is a standard Borel probability 
	space~$(Z,\mu)$ together with a measurable probability measure preserving left-$\Gamma$-action.
\end{defi}

\begin{defi}[parametrised $\ell^1$-norm]
	Let~$M$ be a path-connected, locally path-connected topological space 
	that admits a universal covering space~$\ucov M$, let~$\Gamma:= \pi_1(M)$,
	and let~$\alpha = \Gamma \curvearrowright (Z;\mu)$ be a standard 
	$\Gamma$-space. For~$n\in \N$, we define the \emph{parametrised $\ell^1$-norm}
	\begin{align*}
		|\cdot|_1 \colon C_n (M;\alpha) & \longrightarrow \R_{\geq 0}\\
		\sum_{j=1}^k f_j \otimes \sigma_j & \longmapsto \sum_{j=1}^k \int_Z |f_j|~ d\mu
	\end{align*}
	on the chain complex
	\[ C_n(M;\alpha) := L^\infty (Z,\mu;\Z) \underset{\Z\Gamma}\otimes C_n(\ucov M;\Z),
	\]
	where we assume that~$\sum_{j=1}^k f_j \otimes \sigma_j$ is in reduced form, i.e.,
	all the singular simplices~$\sigma_j\in \map(\Delta^n, \ucov M)$ belong to
	different $\Gamma$-orbits.
	We consider the right-$\Gamma$-action on~$L^\infty (Z,\mu;\Z)$ given by
	\[ (f \cdot \gamma)(x) := f(\gamma \cdot x)
	\]
	for all~$f\in L^\infty (Z,\mu;\Z)$,~$\gamma\in \Gamma$ and~$x\in Z$
	and the left-$\Gamma$-action on~$C_n(\ucov M;\Z)$ induced by the deck
	transformation action of~$\Gamma$ on~$\ucov M$.
	In the following, we also write~$L^\infty (Z;\Z)$ or~$L^\infty(\alpha;\Z)$ 
	for~$L^\infty (Z,\mu;\Z)$.
\end{defi}


\subsection{Parametrised fundamental cycles}
\label{subsec:parfundcyc}

In this subsection we recall the definition of parametrised relative fundamental 
cycles~\cite[Section~10.1]{varubc} and introduce a local criterion for these cycles.

Let~$n\in \N$. Let~$M$ be an oriented compact connected $n$-manifold. We write~$\Gamma := \pi_1(M)$
and~$q\colon \ucov M \longrightarrow M$ for the universal covering of~$M$.
Let~$\alpha = \Gamma \curvearrowright (Z,\mu)$ be a standard $\Gamma$-space.
Since the $\Gamma$-action on~$\ucov M$ restricts to a $\Gamma$-action on~$q^{-1}(\partial M)$,
we can define
\[ D_* := L^\infty (Z;\Z) \underset{\Z\Gamma}{\otimes} C_*\bigl(q^{-1}(\partial M);\Z\bigr)
\]
as a subcomplex of~$C_*(M;\alpha)$. We set
\[ C_* (M,\partial M;\alpha) := C_*(M;\alpha)/D_*
\cong L^\infty (Z;\Z) \underset{\Z\Gamma}{\otimes} C_*\bigl(\widetilde M, q^{-1}(\partial M);\Z\bigr)
\]
and
\[ H_* (M,\partial M;\alpha) := H_* \bigl(C_*(M,\partial M;\alpha)\bigr).
\]

\begin{defi}[parametrised relative fundamental cycle]
	An \emph{$\alpha$-parametris\-ed relative fundamental cycle of~$M$}
	is a chain~$c\in C_n(M;\alpha)$ of the form
	\[ c = c_\Z + \partial b + d
	\]
	with a relative fundamental cycle~$c_\Z \in C_n(M;\Z) \subset C_n(M;\alpha)$,
	a chain~$b\in C_{n+1}(M;\alpha)$ and a chain~$d\in D_n$.
	In other words, $c$ is a cycle in~$C_n(M, \partial M ; \alpha)$
	representing~$[M,\partial M]^\alpha$, i.e., the image of~$[M,\partial M]$
	under the induced map of the inclusion
	\begin{align*}
		C_n(M,\partial M;\Z)
		\cong \Z \underset{\Z\Gamma}\otimes C_n\bigl(\widetilde M, q^{-1} (\partial M);\Z\bigr)
		&\longrightarrow
		C_*(\widetilde M,\partial M;\alpha)
		\\
		1\otimes \sigma &\longmapsto \const_1\otimes \sigma.
	\end{align*}
\end{defi}

\begin{lem}\label{lem:invariants}
	Let~$\Gamma$ be a countable group and let~$\alpha \colon \Gamma \curvearrowright (Z,\mu)$
	be a standard $\Gamma$-space. We write~$A:= L^{\infty}(Z, \mu;\Z)$. Then the canonical map
	$A^\Gamma \longrightarrow A_\Gamma$ from the~$\Gamma$-invariants to the $\Gamma$-coinvariants
	of~$A$ is injective.
\end{lem}

\begin{proof}
	Let~$f\in A^\Gamma \setminus \{0\}$. We want to show that~$[f]\neq 0$ in~$A_\Gamma$.
	Since~$f$ is non-zero there exists a measurable subset~$B\subset X$ with~$\Gamma \cdot B = B$,
	$\mu(B)>0$, and
	\[ f|_B\geq 1 \quad \text{or} \quad f|_B\leq - 1.
	\]
	On the one hand, we have
	\[ \Bigl|\int_B f~d\mu\Bigr| \geq \mu(B) \neq 0.
	\]
	On the other hand, integration~$\int_B \cdot~d\mu \colon A^\Gamma\longrightarrow \R$
	over the $\Gamma$-invariant set~$B$ factors through~$A^\Gamma \longrightarrow A_\Gamma$
	since $\mu$ is $\Gamma$-invariant
	and it follows that~$[f]\neq 0$ in~$A_\Gamma$.
\end{proof}

\begin{defi}[local parametrised fundamental cycles]\label{def:local}
	Let~$M$ be an oriented compact connected $n$-manifold and let~$\Gamma := \pi_1(M)$.
	Let~$U\subset M^\circ$ be an open embedded $n$-ball in the interior of~$M$ 
	with~$\overline U \subset M^\circ$.
	Let~$\alpha\colon \Gamma \curvearrowright (Z,\mu)$ be a standard $\Gamma$-space.
	We write $A := L^\infty(Z,\mu;\Z)$.
	Let~$q\colon \widetilde M \longrightarrow M$ denote the universal covering of~$M$.
	Consider the following chain map
	\begin{align*} g\colon
			A\underset{\Z\Gamma}\otimes C_n\bigl(\widetilde M, q^{-1}(\partial M); \Z\bigr)
			\longrightarrow &A_\Gamma \underset{\Z\Gamma}\otimes 
					C_n\bigl(\widetilde M, q^{-1}(M \setminus U); \Z\bigr)
	\end{align*}
	induced by the inclusion
	\[ \bigl(\widetilde M, q^{-1}(\partial M)\bigr)
			\longrightarrow \bigl(\widetilde M, q^{-1}(M \setminus U)\bigr)
	\]
	and the change of coefficients map corresponding to~$A\longrightarrow A_{\Gamma}$.
	Here, the $\Gamma$-coinvariants $A_{\Gamma}$ of~$A$ are equipped
	with the trivial $\Gamma$-action.
	Rewriting~$g$ gives us a chain map
	\[ f\colon C_n(M,\partial M; \alpha) \longrightarrow 
			A_\Gamma \underset{\Z}\otimes C_n(M, M \setminus U; \Z).
	\]
	Let~$i\colon C_n(M;\alpha)\longrightarrow C_n(M,\partial M;\alpha)$ be the canonical map.
	
	Let~$c\in C_n(M;\alpha)$ be a relative cycle.
	Then~$c$ is called a \emph{$U$-local $\alpha$-parame\-trised relative fundamental cycle of~$M$} if
	\[ F\bigl([i(c)]\bigr) = \const_1 \in A_\Gamma
	    \cong H_n(M, M \setminus U; A_\Gamma),
	\]
	where~$F$ denotes the induced map of~$f$ in homology.
\end{defi}

\begin{prop}[locality of parametrised fundamental cycles]\label{prop:locality}
	Let~$M$ be an oriented compact connected $n$-manifold and let~$\Gamma := \pi_1(M)$.
	Let~$U\subset M^\circ$ be an open embedded $n$-ball with~$\overline U \subset M^\circ$.
	Let~$\alpha\colon \Gamma \curvearrowright (Z,\mu)$ be a standard $\Gamma$-space
	and let~$c\in C_n(M;\alpha)$ be a relative cycle.
	Then the following are equivalent:
	\begin{itemize}
		\item The relative cycle~$c$ is an $\alpha$-parametrised relative fundamental cycle of~$M$.
		\item The relative cycle~$c$ is a $U$-local
				$\alpha$-parametrised relative fundamental cycle of~$M$.
	\end{itemize}
\end{prop}
\begin{proof}
	We write $A := L^\infty(Z,\mu;\Z)$.
	Let
	\[ f\colon C_n(M,\partial M; \alpha) \longrightarrow 
			A_\Gamma \underset{\Z}\otimes C_n(M, M \setminus U; \Z)
	\]
	be given as in Definition~\ref{def:local} and let
	\[i\colon C_n(M;\alpha)\longrightarrow C_n(M,\partial M;\alpha)
	\]
	be the canonical map.
	
	By Lefschetz duality with twisted coefficients, we have
	\[ H_n(M, \partial M;\alpha) \cong H^0(M;\alpha) \cong A^\Gamma,
	\]
	where~$A^\Gamma$ denotes the~$\Gamma$-invariants in~$A$ (the pair~$(M,\partial M)$
	is a connected simple Poincar\'e pair~\cite[Theorem~2.1]{wall} and connected Poincar\'e
	pairs satisfy Lefschetz duality with twisted coefficients~\cite[Lemma~1.2]{wallpc}).
	
	It follows that in homology the map~$F := f_*$ is injective, since it induces the
	canonical projection~$A^\Gamma \longrightarrow A_\Gamma$,
	which is injective by Lemma~\ref{lem:invariants}.
	Furthermore, note that
	\[ F\bigl([M,\partial M]^\alpha\bigr) = \const_1 \in A_\Gamma.
	\]
	Since~$F$ is injective, it follows that~$i(c)$ represents~$[M,\partial M]^\alpha$ if and only
	if~$F\bigl([i(c)]\bigr)=\const_1$. This finishes the proof.
\end{proof}


\subsection{The relative integral foliated simplicial volume}
\label{subsec:ifsv}
We will now introduce (relative) integral foliated simplicial volume.

\begin{defi}[(relative) integral foliated simplicial volume]
	Let~$M$ be an oriented compact connected manifold.
	Then the \emph{(relative) integral foliated simplicial volume of~$M$} is defined by
	\begin{align*}
		\ifsv{M, \partial M} := \inf \bigl\{\ifsv{M,\partial M}^\alpha \bigm| \alpha = \Gamma \curvearrowright (Z,\mu) \text{ is a standard $\Gamma$-space}\bigr\},
	\end{align*}
	where~$\ifsv{M,\partial M}^\alpha 
						:= \inf \bigl\{ |c|_1 \bigm| c \text{ represents } [M,\partial M]^\alpha\bigr\}$.
	If~$\partial M = \emptyset$, we write~$\ifsv{M}^\alpha := \ifsv{M,\partial M}^\alpha$
	and~$\ifsv M := \ifsv{M,\partial M}$.
\end{defi}

The integral foliated simplicial volume yields an upper bound for the $L^2$-Betti numbers; more precisely
\[ \sum_{k=0}^n b_k^{(2)}(\ucov M) \leq (n+1) \cdot \ifsv M
\]
holds~\cite{schmidt} (the original constant $2^{n+1}$ can easily be improved to~$n+1$).
Furthermore, it fits into the following sandwich~\cite{schmidt}
\[ \|M\|\leq \ifsv M \leq \|M\|_\Z,
\]
where~$\|M\|_\Z$ is the \emph{integral simplicial volume} which is given by the
minimal $\ell^1$-norm of integral fundamental cycles of~$M$.
The integral foliated simplicial volume is known to be equal to the simplicial volume
in the case of oriented closed connected hyperbolic 3-manifolds~\cite{loehpagliantini}.
However, it is strictly greater than the simplicial volume for oriented closed connected
hyperbolic $k$-manifolds with~$k\in \N_{\geq 4}$~\cite[Theorem~1.8]{flps}.
Moreover, the integral foliated simplicial volume of oriented closed connected
aspherical manifolds with amenable fundamental group is zero~\cite[Theorem~1.9]{flps}.


\subsection{Normed chain complexes and the parametrised uniform boundary condition for~$S^1$}
	In the following we discuss our main tool for the proof of Theorem~\ref{thm:mainthm},
	the parametrised uniform boundary condition for~$S^1$.

\begin{defi}[(semi-)normed abelian groups]
	Let~$A$ be an abelian group.
	\begin{itemize}
		\item A \emph{semi-norm on~$A$} is a map~$|\cdot|\colon A \longrightarrow \R_{\geq 0}$ with
					the following properties:
			\begin{itemize}
				\item We have~$|0|=0$.
				\item For all~$a\in A$ we have~$|-a|=|a|$.
				\item For all~$a,b\in A$ we have~$|a+b|\leq |a|+|b|$.
			\end{itemize}
		\item A \emph{norm on~$A$} is a semi-norm~$|\cdot|$ on~$A$ such that for all~$a\in A$ we
					have~$|a|=0$ if and only if~$a=0$.
		\item A \emph{(semi-)normed abelian group} is an abelian group equipped with a (semi-)norm.
		\item A homomorphism~$f\colon A\longrightarrow B$ of (semi-)normed abelian groups is called
					\emph{bounded} if there exists a constant~$C\in \R_{>0}$ such that for all~$a\in A$
					we have
					\[ |f(a)|\leq C\cdot |a|.
					\]
	\end{itemize}
\end{defi}

\begin{defi}[normed chain complex]
	A \emph{normed chain complex} is a chain complex in the category of normed abelian
	groups and bounded homomorphisms.
\end{defi}

\begin{exa}
	The singular chain complex together with the $\ell^1$-norm as well as the singular chain complex
	with twisted coefficients together with the parametrised $\ell^1$-norm are normed chain 
	complexes.
\end{exa}

\begin{defi}[UBC]
	Let~$n\in \N$. A normed chain complex~$C_*$ satisfies the \emph{uniform boundary condition in
	degree~$n$},	short $n$-UBC, if there exist a constant~$C\in \R_{>0}$ such that for every 
	null-homologous cycle~$c\in C_n$ there exists an efficient filling, i.e., a
	chain~$b\in C_{n+1}$ with~$\partial b=c$ and~$|b|\leq C\cdot |c|$.
\end{defi}

\begin{prop}[UBC and homotopy]\label{prop:ubc}
	Let~$f_*\colon C_*\longrightarrow D_*$ be a chain homotopy equivalence of normed chain complexes with chain homotopy inverse~$g_*$ and chain homotopies~$h^C_*$ from~$\id_C$ to~$g\circ f$ and~$h^D_*$ from~$\id_D$ to~$f\circ g$. If~$f_n, g_{n+1}$ and~$h^C_n$ are bounded for some~$n\in \N$, then the following holds:
	If~$D_*$ satisfies $n$-UBC, then so does~$C_*$.
\end{prop}

\begin{proof}
	Let~$n\in N$ and let~$K_n$ be the maximum of the bounds of~$f_n, g_{n+1}, h^C_n$ and the $n$-UBC constant of~$D_*$. Let~$c\in C_n$ be a null-homologous cycle. Then~$f_n(c)\in D_n$ is a null-homologous cycle and therefore there exists an efficient filling~$b'\in D_{n+1}$ of~$c$.
	We set~$b:= g_{n+1}(b') - h_n^C(c)$. Then we have
	\begin{align*}
		\partial b = \partial g_{n+1} (b')- \partial h^C_n (c)
								=g_n\circ f_n (c) + h^C_{n-1} (\partial c) - g_n\circ f_n (c) + c=c.
	\end{align*}
	and
	\[ |b|\leq K_n^3\cdot |c| + K_n \cdot |c|.
	\]
	Hence, the chain~$b$ is an efficient filling of~$c$.
\end{proof}

In our proof of Theorem~\ref{thm:mainthm} (Section~\ref{sec:proof}),
the following result will play an important role.
It is a special case of the parametrised uniform boundary condition
for tori~\cite[Theorem~1.3]{varubc}.

\begin{thm}[parametrised UBC for~$S^1$]
\label{thm:ubc}
	Let $\Gamma := \pi_1(S^1) \cong \Z$, and 
  let~$\alpha=\Gamma \curvearrowright (Z,\mu)$ be an essentially free standard $\Gamma$-space.
	Then~$C_*(S^1;\alpha)$
	satisfies the uniform boundary condition in every degree.
\end{thm}


\section{Yano's construction in the parametrised world}
\label{sec:prep}

In addition to Yano's setup (Section~\ref{sec:setup}), we need some technical prerequisites that we cover in the following section.

Let~$M$ be an oriented compact connected smooth $n$-mani\-fold that
admits a smooth $S^1$-action without fixed points and such that the inclusion of every
orbit is $\pi_1$-injective. We write~$\Gamma:= \pi_1(M)$.

\begin{prop}
\label{prop:pioneinjective}
	Let~$k\in \{1, \dots, n-2\}$ and let~$j_1, \dots, j_k \in \{0, \dots, n-3\}$ be pairwise distinct.
	Then, the inclusions~$X_{j_1, \dots, j_k} \subset M_{j_1}$
	and~$X_{j_1, \dots, j_k, -1} \subset M_{j_1}$ are $\pi_1$-injective
	(for any chosen basepoints).
\end{prop}
\begin{proof}
	By Proposition~\ref{prop:boundary}, it is enough to show that~$X_{j_1, \dots, j_k}
	\subset M_{j_1}$ is $\pi_1$-injective.
	By Lemma~\ref{lem:trivialbundle} we have
	\[ X_{j_1,\dots, j_k} \cong \overline X_{j_1,\dots, j_k} \times S^1
	\]
	and each component of~$\overline X_{j_1,\dots, j_k}$ is contractible.
	With this in mind, observe that the composition of maps
	\[ X_{j_1,\dots, j_k} \subset M_{j_1} \stackrel{p_{j_1,0}}{\longrightarrow} M
	\]
	is the inclusion of~$S^1$-orbits into~$M$ and thus $\pi_1$-injective.
	Hence, the inclusion~$X_{j_1,\dots, j_k} \subset M_{j_1}$ is also $\pi_1$-injective.
\end{proof}

\begin{setup}
\label{setup}
Let~$\alpha\colon \Gamma \curvearrowright (Z,\mu)$ be an essentially free
standard $\Gamma$-space. We define a sequence
\begin{align*}
	L^\infty(\alpha_{n-2},\Z) \underset{\Z \Gamma_{n-2}}{\otimes} C_*( \ucov M_{n-2};\Z)
	\stackrel{P_{n-3}}{\longrightarrow} \dots \stackrel{P_{0}}{\longrightarrow}
	L^\infty(\alpha_{0},\Z) \underset{\Z \Gamma_{0}}{\otimes} C_*( \ucov M_{0};\Z)
\end{align*}
of chain maps:
We write~$\alpha_0 := \alpha$, $\alpha_{-1} := \alpha$ and~$\Gamma_j := \pi_1(M_j)$.
For all~$j\in \{1, \dots, n-2\}$ let~$\alpha_j \colon \Gamma_j \curvearrowright (Z,\mu)$
be the $\Gamma_j$-space obtained by restricting~$\alpha$ along~$\pi_1(p_{j,0})$, i.e.,
we consider the $\Gamma_j$-action on~$Z$ given by
\[ \gamma \cdot z := \pi_1(p_{j,0}) (\gamma) \cdot z.
\]
Let~$x_0\in M_{n-2}$. Then, for all~$j\in\{0, \dots, n-3\}$ the map~$p_j$ induces a homomorphism
\begin{align*}
	P_j \colon L^\infty(\alpha_{j+1},\Z) \underset{\Z \Gamma_{j+1}}{\otimes} C_*( \ucov M_{j+1};\Z)
	&\longrightarrow
	L^\infty(\alpha_{j},\Z) \underset{\Z \Gamma_{j}}{\otimes} C_*( \ucov M_{j};\Z),\\
	f \otimes \sigma & \longmapsto f \otimes \ucov p_j \circ \sigma
\end{align*}
where~$\ucov p_j$ denotes the lift of~$p_j$ with respect to the base point~$p_{n-2,j}(x_0)$.
Observe that~$\|P_j\|\leq 1$ holds for all~$j\in\{0, \dots, n-3\}$ with respect to the parametrised~$\ell^1$-norm.
For all~$j,j'\in \{ 0, \dots, n-2\}$ with~$j'<j$ we write
\[ P_{j,j'} := P_{j'} \circ \dots \circ P_{j-1}\colon 
		C_*(M_j,\alpha_j) \longrightarrow C_*(M_{j'},\alpha_{j'}).
\]

For all~$j\in \{0, \dots, n-2\}$ let~$q_j \colon \ucov M_j \longrightarrow M_j$ denote the universal covering of~$M_j$.
Since~$q_j^{-1}(X_j)$ is closed under the $\Gamma_j$-action on~$\ucov M_j$, we can consider the subcomplex
\[ L^\infty(\alpha_{j},\Z) \underset{\Z \Gamma_{j}}{\otimes} C_*( q_j^{-1}(X_j);\Z)
\]
of~$C_*(M_j; \alpha_j)$. Recall that~$X_j\subset M_j$ is $\pi_1$-injective
by Proposition~\ref{prop:pioneinjective} and 
that~$X_j \cong \overline{X_j} \times S^1$ by Lemma~\ref{lem:trivialbundle}, 
where each component of~$\overline{X_j}$ is contractible.
Therefore, the induced inclusion~$\pi_1(X_j) \subset \Gamma_j$ 
does not depend on the chosen basepoint: 
Let~$x,y\in X_j$. Then we choose an embedded path~$\overline\gamma$ in~$\overline M_j$ 
from~$\overline x$ to~$\overline y$ such that the interior of~$\overline\gamma$ does not
intersect the $(n-3)$-skeleton of~$\overline M_j$. Then by construction, the preimage~$\gamma$
of~$\overline \gamma$ under the orbit map is a principal $S^1$-bundle over an embedded interval
and therefore~$\gamma$ is an embedded annulus and the inclusions of the orbits corresponding 
to the start and endpoint of~$\overline \gamma$ are homotopic.
We write~$\Lambda_j := \pi_1(X_j) \subset \Gamma_j$
and~$\alpha'_j:=\res^{\Gamma_j}_{\Lambda_j} \alpha_{j}$.
For every connected component of~$X_j$ we choose a connected component of~$q_j^{-1}$(Y).
Let~$\ucov X_j \subset \ucov M_j$ be the union of the chosen components.
Then there is a canonical (isometric) chain isomorphism
\begin{align*}
	L^\infty(\alpha_{j},\Z) \underset{\Z \Gamma_{j}}{\otimes} C_*( q_j^{-1}(X_j);\Z)
	&\cong L^\infty(\alpha_{j},\Z) \underset{\Z \Gamma_{j}}{\otimes} \Z \Gamma_j \underset{\Z \Lambda_j}{\otimes} C_*( \ucov X_j;\Z)\\
	&\cong L^\infty(\alpha'_{j},\Z) \underset{\Z \Lambda_j}{\otimes} C_*( \ucov X_j;\Z)\\
	&= C_* (X_j; \alpha'_{j}).
\end{align*}

The analogous statements with~$X_{j_1, \dots, j_k}$ (or~$X_{j_1, \dots, j_k,-1}$) replacing~$X_j$ also hold for all~$k\in \{1, \dots, n-2\}$ and all~$j_1, \dots, j_k \in \{0, \dots, n-3\}$ that
are pairwise distinct,
where we define
\[ \Lambda_{j_1,\dots, j_k (, -1)} := \pi_1(X_{j_1,\dots, j_k (, -1)}) \subset \Gamma_{j_1}
		\quad \text{and} \quad
	 \alpha'_{j_1,\dots, j_k(, -1)} := 
			\res^{\Gamma_{j_1}}_{\Lambda_{j_1,\dots,j_k(, -1)}} \alpha_{j_1}.
\]

Furthermore, note that~$\alpha'_{j_1,\dots, j_k(, -1)}$ is essentially free and 
that the normed chain 
complex~$C_*\bigl(X_{j_1,\dots, j_k (, -1)};\alpha_{j_1,\dots, j_k (, -1)}\bigr)$ satisfies UBC in every degree by Lemma~\ref{lem:trivialbundle} and applying Theorem~\ref{thm:ubc} 
on every component of~$X_{j_1,\dots, j_k (, -1)}$.
\end{setup}


\section{Constructing parametrised chains with small norm}
\label{sec:constructchain}

In this section we construct
parametrised chains of small norm that we adjust
in Section~\ref{sec:proof} to get parametrised 
fundamental cycles of small norm.

\begin{prop}
\label{prop:trivialbundle}
	Let~$M$ be an oriented compact connected smooth $n$-manifold that
	admits a smooth $S^1$-action without fixed points and such that the inclusion of every
	orbit is $\pi_1$-injective. We write~$\Gamma:=\pi_1(M)$.
	Let~$\alpha\colon \Gamma \curvearrowright (Z,\mu)$ be an essentially free
	standard $\Gamma$-space.
	Then, we have
	\[ \ifsv{M_{n-2},\partial M_{n-2}}^{\alpha_{n-2}}=0,
	\]
	where~$M_{n-2}$ is defined as in Section~\ref{sec:setup}
	and~$\alpha_{n-2}$ as in Setup~\ref{setup}.
\end{prop}

In fact, the proof will give an explicit construction of efficient
parame\-trised relative fundamental cycles:

\begin{proof}
	Recall that we have~$M_{n-2}\cong \overline M_{n-2} \times S^1$ by Lemma~\ref{lem:trivialbundle}.
	In particular, $\overline M_{n-2}$ is an orientable compact connected
	$(n-1)$-manifold. We choose an orientation on~$\overline M_{n-2}$ such that
	the homeomorphism~$M_{n-2}\cong \overline M_{n-2} \times S^1$ is orientation-preserving.
	Let~$\Lambda\subset \Gamma_{n-2}$ be the subgroup corresponding to the $S^1$-factor
	of~$M_{n-2}$. 
	Then, the composition of maps
	\[ \Lambda \subset \Gamma_{n-2} \longrightarrow \Gamma
	\]
	is injective as it is induced by the inclusion of an $S^1$-orbit.
	Hence, $\alpha' := \res^{\Gamma_{n-2}}_\Lambda \alpha_{n-2}$
	is an essentially free standard $\Lambda$-space and the result follows~\cite[Lemma~10.8]{varubc}.
	More precisely, let~$\overline K$ be a triangulation of~$\overline M_{n-2}$
	that extends the simplicial structure of~$\partial \overline M_{n-2}$
	from Remark~\ref{rem:complex}.
	Since~$\overline M_{n-2}$ is an oriented compact connected manifold, we can construct
	a relative fundamental cycle~$\overline z \in C_{n-1}(\overline M_{n-2};\Z)$ out of
	the triangulation~$\overline K$ of~$\overline M_{n-2}$.
	Then, for all~$\epsilon \in \R_{>0}$ we can find an $\alpha'$-parametrised 
	fundamental cycle~$c_{S^1}\in C_1(S^1;\alpha')$ such that the $\alpha_{n-2}$-parametrised 
	relative fundamental cycle given by
	\[ z:= \overline z \times c_{S^1} \in C_n(M_{n-2};\alpha_{n-2})
	\]
	has $\ell^1$-norm less than~$\epsilon$.
\end{proof}


\section{Proof of Theorem~\ref{thm:mainthm}}
\label{sec:proof}

In this last section we prove Theorem~\ref{thm:mainthm}.
We basically transfer Yano's original proof~\cite[Section~3]{yano} to the parametrised setting
with the difference that we use the uniform boundary condition for~$S^1$
(Theorem~\ref{thm:ubc}) to get efficient fillings.

In the following, we use the same notation as in Section~\ref{sec:setup} and Setup~\ref{setup}.
Let~$\epsilon \in \R_{>0}$.  We start with an $\alpha_{n-2}$-parametrised relative fundamental cycle~$z= \overline z \times c_{S^1}$ of~$M_{n-2}$ as in the proof of
Proposition~\ref{prop:trivialbundle} with $\ell^1$-norm less than~$\epsilon$.
For all~$j \in \{ -1, \dots, n-3 \}$ we define
\[ \overline z_j := (\partial \overline z )|_{\overline N_j} \in C_{n-2}(\overline N_j;\Z)
\]
as the sum of all simplices in~$\partial \overline z$ that belong to the
subcomplex~$\overline N_j\subset \partial \overline M_{n-2}$. We set
\[ z_j := \overline z_j \times c_{S^1} \in L^\infty(\alpha_{n-2},\Z) \underset{\Z \Gamma_{n-2}}{\otimes} C_{n-1}\bigl( q_{n-2}^{-1}(\ucov N_j);\Z\bigr).
\]
Analogously, for all~$k \in \{ 1, \dots, n-1 \}$ and 
all~$j_1, \dots, j_k \in \{ -1, \dots, n-3 \}$ 
that are pairwise distinct we define inductively
\[ \overline z_{j_1, \dots, j_k} := 
		(\partial \overline z_{j_1, \dots, j_{k-1}}) |_{\overline N_{j_1, \dots, j_k}} 
		\in C_{n-1-k}( \overline N_{j_1, \dots, j_k};\Z)
\]
and we set~$\overline z_{j_1, \dots, j_k}:=0$ if~$j_1, \dots, j_k$ are not pairwise distinct.
We define
\[ z_{j_1, \dots, j_k} := \overline z_{j_1, \dots, j_k} \times c_{S^1}
		\in L^\infty(\alpha_{n-2},\Z) \underset{\Z \Gamma_{n-2}} \otimes
		C_{n-k}\bigl( q_{n-2}^{-1}(\ucov N_{j_1, \dots, j_k});\Z\bigr).
\]

\begin{lem}
\label{lem:boundary}
	We have
	\[ \partial z = \sum_{j=-1}^{n-3} z_j
			\quad \text{ and } \quad
		 \partial z_{j_1, \dots, j_k} = \sum_{j=-1}^{n-3} z_{j_1, \dots, j_k, j}
	\]
	for all~$k \in \{ 1, \dots, n-1\}$ and pairwise 
	distinct~$j_1, \dots, j_k \in \{ -1, \dots, n-3 \}$.
\end{lem}
\begin{proof}
	It is is enough to show the 
	analogous statements for~$\overline z$ and~$\overline z_{j_1,\dots, j_k}$.
	Since~$\partial\overline M_{n-2}$
	is a subcomplex of~$\overline K$ it follows from Remark~\ref{rem:complex}
	that we have
	\[ \partial \overline z
			= (\partial \overline z)|_{\partial \overline M_{n-2}}
			= \sum_{j=-1}^{n-3} (\partial \overline z)|_{\overline N_j}
			= \sum_{j=-1}^{n-3} \overline z_j
	\]
	and for all~$k\in\{1,\dots, n-1\}$ and all~$j_1,\dots, j_k\in\{-1,\dots, n-3\}$
	that are pairwise distinct, we have
	\begin{align*}
		\partial \overline z_{j_1,\dots, j_k}
			= (\partial \overline z_{j_1,\dots, j_k})|_{\partial \overline N_{j_1, \dots, j_k}}
			= \sum_{i}(\partial \overline z_{j_1,\dots, j_k})|_{\overline N_{j_1,\dots, j_k,i}}
			= \sum_{j=-1}^{n-3} \overline z_{j_1,\dots, j_k,j}
	\end{align*}
	where~$i$ ranges over~$\{-1,\dots, n-3\}\setminus\{j_1,\dots , j_k\}$.
\end{proof}

\begin{lem}
\label{lem:alternating}
	Let~$k\in \{1,\dots, n-1\}$ and
	let~$\tau$ be a permutation of~$\{ 1, \dots, k \}$. Then we have
	\[ z_{j_1, \dots, j_k} = \sign(\tau) \cdot z_{j_{\tau(1)}, \dots, j_{\tau(k)}}.
	\]
\end{lem}
\begin{proof}
	We may assume that~$\tau$ is a transposition.
	In fact, it is enough to consider the case of swapping the last two indices, i.e., to show
	that
	\[ z_{j_1, \dots, j_{k}} = - z_{j_1, \dots, j_{k-2}, j_k, j_{k-1}}.
	\]
	By Lemma~\ref{lem:boundary} we have
	\begin{align*}
		0 = \partial \partial z_{j_1, \dots, j_{k-2}} = \sum_{j=-1}^{n-3} \partial z_{j_1, \dots, j_{k-2}, j}
			=\sum_{j=-1}^{n-3} \sum_{i=-1}^{n-3} z_{j_1, \dots, j_{k-2}, j, i}.
	\end{align*}
	Because~$\partial \overline M_{n-2}$ is a subcomplex of~$\overline K$
	and by Remark~\ref{rem:complex} it follows from the definition of~$z_{j_1, \dots, j_{k}}$
	that the only term that can cancel~$z_{j_1, \dots, j_{k}}$ out
	is a term that has the same indices as~$z_{j_1, \dots, j_{k}}$, 
	namely~$z_{j_1, \dots, j_{k-2}, j_k, j_{k-1}}$,
	and therefore,
	\[ z_{j_1, \dots, j_{k}} = - z_{j_1, \dots, j_{k-2}, j_k, j_{k-1}}.\qedhere
	\]
\end{proof}

\begin{lem}
	There exist families of chains
	\[ w_{j_1, \dots, j_k}
	\in C_{n-k+1}(X_{j_1, \dots, j_k};\alpha'_{j_1, \dots, j_k})
	\]
	and
	\[ w_{j_1, \dots, j_k, -1}
	\in C_{n-k}(X_{j_1, \dots, j_k,-1};\alpha'_{j_1, \dots, j_k,-1})
	\]
	for all~$k\in \{1, \dots, n-2\}$ and~$j_1, \dots, j_k \in \{0, \dots, n-3\}$ 
	satisfying the following conditions:
	\begin{enumerate}
		\item The chains~$w_{j_1, \dots, j_k}$ and~$w_{j_1, \dots, j_k , -1}$ are alternating 
		with respect to permutations of the indices~$\{j_1, \dots, j_k\}$.
		\item We have
			\[ \partial w_{j_1, \dots, j_{n-2}, -1} = P_{n-2,j_1} (z_{j_1, \dots, j_{n-2}, -1})
			\]
			and
			\[ \partial w_{j_1, \dots, j_k}
					= P_{n-2,j_1}(z_{j_1, \dots, j_k}) - \sum_{j=-1}^{n-3} w_{j_1, \dots, j_k, j}
			\]
			and for~$k< n-2$ we have
			\[ \partial w_{j_1, \dots, j_k, -1}
			    = P_{n-2,j_1}(z_{j_1, \dots, j_k, -1}) + \sum_{j=0}^{n-3} w_{j_1, \dots, j_k, j, -1}.
			\]
		\item Let~$C \in \R_{>0}$ be the maximum of UBC-constants for all~$X_{j_1,\dots,j_k}$
		and all~$X_{j_1,\dots,j_k, -1}$ in all degrees from~$0$ to~$n$.
		Then, we have
			\[ |w_{j_1, \dots, j_k}|_1 \leq C \cdot B^{n-k-1} \cdot (n+1)! \cdot |z|_1
			\]
		where~$B:= 1+C \cdot (n-1)$ and~$j_k$ might be~$-1$.
	\end{enumerate}
\end{lem}

\begin{proof}
	We prove the lemma by downward induction on~$k$. Let~$k=n-2$.
	By Lemma~\ref{lem:boundary} we have
	\[ \partial z_{n-3, \dots, 0} = \sum_{j=-1}^{n-3} z_{n-3, \dots, 0, j} = z_{n-3, \dots, 0, -1}
	\]
	We even have
	\[ \partial \bigl( P_{n-3}(z_{n-3, \dots, 0})|_{X_{n-3, \dots, -1}} \bigr) 
			= P_{n-3} \bigl(z_{n-3, \dots, 0, -1}\bigr)
	\]
	since~$X_{n-3, \dots, -1}$ is a union of connected components of $X_{n-3, \dots, 0}$
	by Proposition~\ref{prop:boundary},
	and therefore,
	\[ P_{n-3}(z_{n-3, \dots, -1}) \in 
	     C_{1} (X_{n-3, \dots, -1} ; \alpha'_{n-3, \dots, -1} )
	\]
	is a null-homologous cycle.
	
	We can apply the parametrised uniform boundary condition for~$S^1$ 
	(Theorem~\ref{thm:ubc}) on each connected component of~$X_{n-3, \dots, -1}$.
	Then, there exists a 
	chain~$w_{n-3, \dots, -1} \in C_{2} (X_{n-3, \dots, -1} ; \alpha'_{n-3, \dots, -1} )$ with
	\[ \partial w_{n-3, \dots, -1} = P_{n-3}(z_{n-3, \dots, -1})
			\quad \text{and} \quad
		 |w_{n-3, \dots, -1}|_1 \leq C \cdot |z_{n-3, \dots, -1}|_1.
	\]
	For each permutation~$\tau$ of~$\{0,\dots, n-3\}$ we set
	\[ w_{\tau(n-3), \dots, \tau(0), -1} := \sign (\tau) \cdot P_{n-3,\tau(n-3)} (w_{n-3, \dots, 0, -1})
	\]
	and we set~$w_{j_1,\dots, j_{n-2}, j} = 0$ for~$\{j_1,\dots, j_{n-2}, j\}\neq \{-1,\dots, n-3\}$.
	
	Now, let~$k\in \{1, \dots, n-2\}$ such that~$w_{j_1, \dots, j_k, j}$ is
	defined for all~$j_1, \dots, j_k \in \{0, \dots, n-3\}$ and all~$j\in \{-1, \dots, n-3\}$.
	Let~$j_1, \dots, j_k \in \{ 0, \dots, n-3 \}$ with~$j_1 > \dots > j_k$.
	We want to define~$w_{j_1,\dots, j_k}$.
  We observe that
	\[ \widetilde z_{j_1, \dots, j_k} :=
	    P_{n-2, j_1}(z_{j_1, \dots, j_k}) - \sum_{j=-1}^{n-3} w_{j_1, \dots, j_k, j}
	\]
	is a cycle in~$C_{n-k}(X_{j_1, \dots, j_k};\alpha'_{j_1, \dots, j_k})$.
	Furthermore, the cycle~$\widetilde z_{j_1, \dots, j_k}$ is null-homologous:
	By Lemma~\ref{lem:trivialbundle}, each component of~$\overline X_{j_1, \dots, j_k}$ is 
	contractible and we have~$X_{j_1, \dots, j_k}\cong \overline X_{j_1, \dots, j_k} \times S^1$
	which implies
	\[ H_l(X_{j_1, \dots, j_k} ; \alpha'_{j_1, \dots, j_k} )\cong 0
	\]
	for all~$l\in \N_{\geq 2}$.
	We can apply the parametrised uniform boundary condition for~$S^1$ 
	(Theorem~\ref{thm:ubc}) on each connected component of~$X_{j_1, \dots, j_k}$.
	Then, there exists a 
	chain~$w_{j_1, \dots, j_k} \in C_{n-k+1} (X_{j_1, \dots, j_k} ; \alpha'_{j_1, \dots, j_k} )$ with
	\[ \partial w_{j_1, \dots, j_k} = \widetilde z_{j_1, \dots, j_k}
			= P_{n-2, j_1}(z_{j_1, \dots, j_k}) - \sum_{j=-1}^{n-3} w_{j_1, \dots, j_k, j}
	\]
	and
  \begin{align*}
		|w_{j_1, \dots, j_k}|_1 &\leq 
							C \cdot |z_{j_1, \dots, j_k}|_1 + C\cdot \sum_{j = -1}^{n-3} 
									|w_{j_1,\dots, j_k,j}|_1\\
						&\leq C \cdot (n+1)! \cdot |z|_1 
						   + (n-1) \cdot C^2 \cdot B^{n-k-2} \cdot (n+1)! \cdot |z|_1\\
						&\leq C\cdot B^{n-k-1} \cdot (n+1)! \cdot |z|_1.
	\end{align*}
	For arbitrary~$j_1, \dots, j_k \in \{ 0, \dots, n-3 \}$ we define 
	$w_{j_1, \dots, j_k} := 0$ if~$j_1, \dots, j_k$ are not pairwise distinct and otherwise we define
	\[ w_{j_1, \dots, j_k} := \sign(\tau) \cdot P_{\tau(j_1),j_1} (w_{\tau(j_1), \dots, \tau(j_k)}),
	\]
	where~$\tau$ is the unique permutation
	on~$\{ j_1, \dots, j_k \}$ with~$\tau(j_1) > \dots > \tau(j_k)$.
	
	Let now~$k\in \{1, \dots, n-3\}$ such that~$w_{j_1, \dots, j_k, j}$ is defined for all~$j_1, \dots, j_k \in \{0, \dots, n-3\}$ and all~$j\in \{0, \dots, n-3\}$.
	Let~$j_1, \dots, j_k \in \{0, \dots, n-3\}$ with~$j_1 > \dots > j_k$.
	We want to define~$w_{j_1, \dots, j_k, -1}$.
	We consider
	\[ \hat{z}_{j_1, \dots, j_k}
			:= P_{n-2,j_1}(z_{j_1, \dots, j_k}) - \sum_{j=0}^{n-3} w_{j_1, \dots, j_k, j}
			\in C_{n-k} ( X_{j_1, \dots, j_k} ; \alpha'_{j_1, \dots, j_k} )
	\]
	and verify that
	\[ \partial \hat{z}_{j_1, \dots, j_k} 
	     = P_{n-2,j_1}(z_{j_1, \dots, j_k, -1}) + \sum_{j=0}^{n-3} w_{j_1, \dots, j_k, j, -1}.
	\]
	Since~$X_{j_1, \dots, j_k, -1}$ is the union of components of~$X_{j_1, \dots, j_k}$
	that lie in~$p_{n-2, j_1}(\widetilde N_{-1})$, we have
	that
	\[ \partial \hat{z}_{j_1, \dots, j_k} \in
	     C_{n-k} ( X_{j_1, \dots, j_k, -1} ; \alpha'_{j_1, \dots, j_k, -1} )
	\]
	is a null-homologous cycle
	and by Theorem~\ref{thm:ubc} there exists an efficient filling
	\[ w_{j_1, \dots, j_k, -1} 
       \in C_{n-k} ( X_{j_1, \dots, j_k, -1} ; \alpha'_{j_1, \dots, j_k, -1} ),
	\]
	i.e., we have
	\[ \partial w_{j_1, \dots, j_k, -1} = \partial \hat{z}_{j_1, \dots, j_k}
			=P_{n-2,j_1}(z_{j_1, \dots, j_k, -1}) + \sum_{j=0}^{n-3} w_{j_1, \dots, j_k, j, -1}
	\]
	and
	\begin{align*}
		|\partial w_{j_1, \dots, j_k, -1}|_1 &\leq C\cdot |z_{j_1, \dots, j_k, -1}|_1
				+C\cdot \sum_{j=0}^{n-3}|w_{j_1, \dots, j_k, j, -1}|_1\\
			&\leq C\cdot (n+1)!\cdot |z|_1 + (n-2) \cdot C^2 \cdot B^{n-k-3} (n+1)! \cdot|z|_1 \\
			&\leq C\cdot B^{n-k-2}\cdot (n+1)! \cdot |z|_1.
	\end{align*}
	For arbitrary~$j_1, \dots, j_k \in \{ 0, \dots, n-3 \}$ we define 
	$w_{j_1, \dots, j_k,-1} := 0$ if~$j_1, \dots, j_k$ are not pairwise distinct
	and otherwise we define
	\[ w_{j_1, \dots, j_k,-1} 
			:= \sign(\tau) \cdot P_{\tau(j_1),j_1} (w_{\tau(j_1), \dots, \tau(j_k), -1}),
	\]
	where~$\tau$ is the permutation
	on~$\{ j_1, \dots, j_k \}$ with~$\tau(j_1) > \dots > \tau(j_k)$.
\end{proof}

Finally, we are prepared to prove Theorem~\ref{thm:mainthm}:

\begin{proof}[Proof of Theorem~\ref{thm:mainthm}]
	We set
	\[ z':= P_{n-2,0} (z) - \sum_{j=0}^{n-3} P_{j,0} (w_j)\in C_n(M;\alpha).
	\]
	By construction, $z'$	is a relative cycle in~$C_n(M;\alpha)$ with norm
	\[ |z'|_1 \leq C \cdot B^{n-3} \cdot (n+1)! \cdot |z|_1,
	\]
	where~$z$ can be chosen with arbitrary small norm.
	
	It is left to show that~$z'$ is an $\alpha$-parametrised relative fundamental cycle of~$M$:
	We write~$p:=p_{n-2,0}$.
	Let~$x\in M\setminus p(\partial M_{n-2})$ and
	let~$U\subset M\setminus p(\partial M_{n-2})$ be an open embedded $n$-ball with~$x \in U$ 
	and~$\overline U \subset M^\circ$.
	By Proposition~\ref{prop:locality} it is sufficient to prove that~$z'$ is a $U$-local
	$\alpha$-parametrised relative fundamental cycle of~$M$, i.e., we have
	\[ F\bigl([i(z')]\bigr) = \const_1\in A_\Gamma \cong H_n(M,M\setminus U, A_\Gamma),
	\]
	where~$F:= f_*$ is the induced map in homology of the chain map
	\[ f\colon C_n(M,\partial M; \alpha) \longrightarrow 
			A_\Gamma \underset{\Z}\otimes C_n(M, M \setminus U; \Z)
	\]
	which is defined as in Definition~\ref{def:local} and
	\[i\colon C_n(M;\alpha)\longrightarrow C_n(M,\partial M;\alpha). 
	\]
	is the canonical map.
	Here, we write~$A:= L^\infty(\alpha;\Z)$.
	
	We write~$\beta:= \alpha_{n-2}$,~$B:= L^\infty(\beta;\Z)$,~$N:=M_{n-2}$ and~$\Lambda:= \pi_1(N)$.
	Since by construction~$p$
	is the identity on~$M\setminus p(\partial M_{n-2})$ we also write~$U$ for~$p^{-1}(U)$.
	Let
	\[ g\colon C_n(N,\partial N; \alpha) \longrightarrow 
			B_\Lambda \underset{\Z}\otimes C_n(N, N \setminus U; \Z)
	\]
	be the analogous map to~$f$ from Definition~\ref{def:local} and let
	\[j\colon C_n(N;\alpha)\longrightarrow C_n(N,\partial N;\alpha)
	\]
	be the canonical map.
	
	Consider the following diagram
	\[ \begin{tikzcd}
		C_n(N;B) \arrow{r}{p_*} \arrow[swap]{d}{j} 
		& C_n(M;A) \arrow{d}{i}\\
		C_n(N, \partial N;B) \arrow[swap]{d}{g} 
		& C_n(M, \partial M;A) \arrow{d}{f}\\
	  B_\Lambda \underset{\Z}\otimes C_n(N, N \setminus U; \Z) \arrow{r}{t}
		& A_\Gamma \underset{\Z}\otimes C_n(M, M \setminus U; \Z)
	\end{tikzcd}
	\]
	In the following capital letters denote the induced maps in homology.
	Since by construction, $z'$ and~$p_*(z)$ coincide on~$U$, we have
	\[ f\circ i (z') = f\circ i \bigl(p_*(z)\bigr).
	\]
	Moreover, $z$ is a $\beta$-parametrised fundamental cycle of~$N$, so by
	Proposition~\ref{prop:locality}, we have
	\[ G\bigl([j(z)]\bigr) = \const_1 \in B_\Lambda.
	\]
	Putting all together, it follows that
	\[ F\bigl([i(z')]\bigr) = [f\circ i \circ p_*(z)] = T\circ G\bigl([j(z)]\bigr)
			= T(\const_1) = \const_1\in A_\Gamma
	\]
	and it follows from Proposition~\ref{prop:locality} that~$z'$ is an
	$\alpha$-parametrised fundamental cycle.
\end{proof}

\begin{rem}[essentially free]
	Note that we never used that the whole action~$\Gamma\curvearrowright Z$ is essentially
	free, but only that the restrictions of the action to every ($\pi_1$-injective) 
	orbit~$S^1\cdot x$ on~$Z$ are
	essentially free (in Proposition~\ref{prop:trivialbundle}
	as well as the inductive filling argument using
	UBC for~$S^1$ (Theorem~\ref{thm:ubc}) in the proof of Theorem~\ref{thm:mainthm}).
\end{rem}


\medskip
\vfill

\noindent
\emph{Daniel Fauser}\\[.5em]
  {\small
  \begin{tabular}{@{\qquad}l}
    Fakult\"at f\"ur Mathematik, 
    Universit\"at Regensburg, 
    93040 Regensburg\\
    \textsf{daniel.fauser@gmx.de},
		\textsf{http://www.mathematik.uni-r.de/fauser}
  \end{tabular}}
\end{document}